\documentclass{amsart}
\usepackage{amssymb}
\usepackage{epsfig}
\usepackage{fancyhdr}

\def\Div{{\rm div\,}}

\def\T{\rm T}
\def\D{\rm D}
\def\R{\rm R}

\newcommand{\al}{\alpha}

\newcommand{\ga}{\gamma}
\newcommand{\de}{\delta}

\newcommand{\eps}{\varepsilon}
\newcommand{\va}{\varphi}

\def\si{\sigma}
\def\De{\Delta}

\def\nb{\nabla}
\def\eps{\varepsilon}
\def\Om{\Omega}
\def\pa{\partial}
\def\Oi{\int_{\Omega}}

\def\ra{\rightarrow}

\newcommand{\benn}{\begin{eqnarray*}}
\newcommand{\eenn}{\end{eqnarray*}}
\newcommand{\ben}{\begin{eqnarray}}
\newcommand{\een}{\end{eqnarray}}
\newcommand{\bal}{\begin{aligned}}
\newcommand{\eal}{\end{aligned}}

\theoremstyle{plain}

\newtheorem{lemma}{Lemma}
\newtheorem{theorem}{Theorem}

\newtheorem{definition}[lemma]{Definition}

\def\dis{\displaystyle}

\def\\{\hfil\break}
\def\const{{\rm const}}
\def\D{{\Bbb D}}
\def\I{{\Bbb I}}
\def\N{{\Bbb N}}
\def\R{{\Bbb R}}
\def\T{{\Bbb T}}

\def\d{\tilde{d}}

\title[Large flux for NSE]{Existence of global weak solutions \\ for Navier-Stokes equations with large flux}

\author[J. Renc{\l}awowicz \& W. M.
Zaj\c{a}czkowski]{Joanna Renc{\l}awowicz \& Wojciech M.
Zaj\c{a}czkowski}
\address{Joanna Renc{\l}awowicz:  Institute of Mathematics, Polish Academy of
Sciences, \'{S}nia\-dec\-kich 8, 00-956 Warsaw, Poland, e-mail:
jr@impan.gov.pl}
\address{Wojciech M. Zaj\c{a}czkowski:  Institute of Mathematics, Polish Academy of
Sciences, \'{S}niadeckich 8, 00-956 Warsaw, Poland, e-mail:
wz@impan.gov.pl and Institute of Mathematics and Cryptology,
Military University of Technology, Kaliskiego 2, 00-908 Warsaw,
Poland}

\thanks{Research supported by MNiSW grant no N N201 396937}
\date{March 24, 2010}

\begin{document}

\begin{abstract}
Global existence of weak solutions to the Navier-Stokes equation
in a cylindrical domain under the slip boundary conditions and
with inflow and outflow was proved. To prove the energy estimate,
crucial for the proof, we use the Hopf function. This makes us
possible to derive such estimate that the inflow and outflow must
not vanish as $t\ra\infty.$ The proof requires estimates in
weighted Sobolev spaces for solutions to the Poisson equation.
Finally, the paper is the first step to prove the existence of
global regular special solutions to the Navier-Stokes equations
with inflow and outflow.
\end{abstract}

\subjclass[2000]{Primary 35Q30; Secondary 76D03, 76D05}

\keywords{Navier-Stokes equation, weighted Sobolev spaces, Neumann boundary-value problem, Dirichlet boundary-value problem, global solutions, large flux}

\maketitle

\section{Introduction}

We consider viscous incompressible fluid motion in a finite
cylinder with large inflow and outflow, assuming boundary slip
conditions. Hence, the following initial boundary value problem is
examined.
\ben \label{NS} \bal %
&v_{t}+v\cdot\nabla v-\Div \T(v,p)=f\quad &{\rm in}\ \
 \Omega^T=\Omega\times(0,T),\\
&\Div v=0\quad &{\rm in}\ \ \Omega^T,\\
 &v\cdot\bar n=0\quad &{\rm
on}\  S_1^T ,\\
&\nu \bar n\cdot\D(v)\cdot\bar\tau_\alpha + \ga v \cdot
\bar\tau_{\al}=0,\ \ \alpha=1,2,\quad
 &{\rm on}\ \ S_1^T,\\
 &v\cdot\bar n=d\quad &{\rm
on}\  S_2^T ,\\
 & \bar n\cdot\D(v)\cdot\bar\tau_\alpha =0,\ \
\alpha=1,2,\quad
 &{\rm on}\ \ S_2^T,\\
&v\big|_{t=0}=v(0)\quad &{\rm in}\ \ \Omega, \eal \een \noindent
where $\Omega\subset\R^3$ is a cylindrical domain,
$S=\partial\Omega$, $v$ is the velocity of the fluid motion with
\mbox{$v(x,t)=(v_1(x,t),v_2(x,t),v_3(x,t))\in\R^3$} ,
$p=p(x,t)\in\R^1$ denotes the pressure,
$f=f(x,t)=(f_1(x,t),f_2(x,t),f_3(x,t))\in\R^3$ -- the external
force field, $x=(x_1, x_2, x_3)$ are the Cartesian coordinates,
$\bar n$ is the unit outward vector normal to the boundary $S$ and
$\bar\tau_\alpha$, $\alpha=1,2,$ are tangent vectors to $S$ and
$\cdot$ denotes the scalar product in $\R^3$.
 We define the stress tensor $\T(v,p)$ as
$$
\T(v,p)=\nu\D(v)-p\I,
$$
where $\nu$ is the constant viscosity coefficient and  $\I$ is the
unit matrix. Next, $\ga >0$ is the slip coefficient and $\D(v)$
denotes the dilatation tensor of the form
$$
\D(v)=\{v_{i,x_j}+v_{j,x_i}\}_{i,j=1,2,3}.$$

We assume that $\Omega\subset\R^3$ is a~cylindrical type domain
parallel to the axis $x_3$ with arbitrary cross section. We set $S
= S_1 \cup S_2$ where $S_1$ is the part of the boundary which is
parallel to the axis $x_3$ and $S_2$ is perpendicular to $x_3$.
Hence \benn S_1 & = & \{x\in\R^3:\varphi_0(x_1,x_2)=c_0,\ -a<x_3<a\}, \\
S_2(-a)& = & \{x\in\R^3:\varphi_0(x_1,x_2)<c_0,\ \ x_3 = -a \}, \\
S_2(a)& = & \{x\in\R^3:\varphi_0(x_1,x_2)<c_0,\ \ x_3 = a \} \eenn
where $a,c_0$ are positive given numbers and
$\varphi_0(x_1,x_2)=c_0$ describes a~sufficiently smooth closed
curve in the plane $x_3=\const.$

\begin{figure}[hbt]
\begin{center}
\epsfig{file=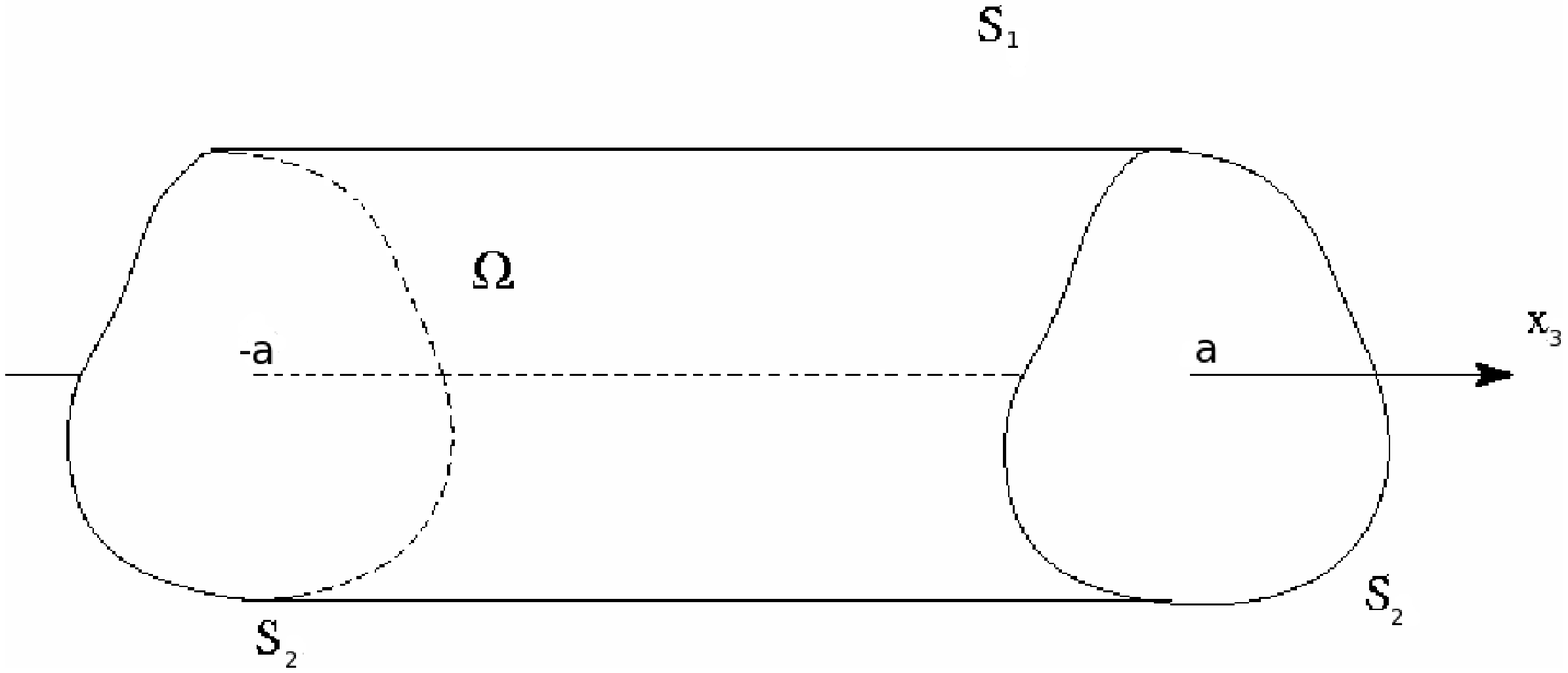,width=4.5 in}
\end{center}
\caption{Domain $\Omega.$}%
\label{om1}
\end{figure}

To describe inflow and outflow we define \ben \label{d} \bal d_1 &
=  -v\cdot \bar{n}|_{S_2(-a)} \\ d_2 & =  v \cdot
\bar{n}|_{S_2(a)} \eal \een with $d_i \ge 0, i=1,2.$ We infer
compatibility conditions \ben \label{d} \int_{S_2(-a)} d_1dS_2 =
\int_{S_2(a)} d_2 dS_2. \een

The aim of this paper is to prove the existence of global weak
solutions to problem (\ref{NS}) without restrictions on magnitudes
of external force $f,$ initial data $v(0),$ inflow $d_1$ and
outflow $d_2.$ We would like to show the existence of such
solutions that the flux does not have to vanish as $t \ra \infty.$
The presented in our paper method
would allow us to prove the existence of global regular solutions
in the cylinder (in the meaning of \cite{RZ3}) which are much more
general than in \cite{K1}, \cite{K2}, \cite{Z} because in these
papers the flux must converge to zero sufficiently fast.

We define a space natural for the study of the weak solutions to the Navier-Stokes equations: %
\benn V_2^0(\Om^T) = \{u: ||u||_{V^0_2(\Om^T)} = {\rm ess}
\sup_{t\in(0,T)}||u||_{L_2(\Om)} + \left(\int_0^T ||\nb
u||_{L_2(\Om)}^2 dt \right)^{1/2} <\infty \}. \eenn

To simplify the notation, we do not distinguish between norms of scalar and vector function and we write $$\|f\|:=\sum_{i=1}^3 \|f_i\|\quad  {\rm for \ any} \quad f=(f_1, f_2, f_3).$$ We also use $$\|d\| :=\|d_1\| + \|d_2\|$$ for inflow $d_1$ and outflow $d_2.$

\begin{theorem}
Assume the compatibility condition (\ref{d}). Assume that $v(0)
\in L_2(\Om),$ $f \in L_2(0,T;L_{6/5}(\Om)),$ $d_i \in
L_{\infty}(0,T; W^{s-1/p}_p(S_2)) \bigcap
L_2(0,T;W^{1/2}_2(S_2)),$ $\frac{3}{p} +\frac{1}{3} \le s, p>3$ or
$p=3, s> \frac{4}{3},$ $d_{i,t} \in L_2(0,T;W^{1/6}_{6/5}(S_2)),
i=1,2.$
 Then there exists a weak solution $v$ to problem (\ref{NS}) such
that $v$ is weakly continuous with respect to $t$ in $L^2(\Om)$ norm and
$v$ converges to $v_0$ as $t\ra 0$ strongly in $L^2(\Om)$ norm.
Moreover, $v\in V^0_2(\Om^T),$ $v\cdot\bar{\tau}_{\al} \in
L_2(0,T;L_2(S_1)), \al =1,2,$ and $v$ satisfies \ben \label{est-v}
\bal & \|v\|^2_{V^0_2(\Om^t)} + \ga \sum_{\al=1}^2 \int_0^t
\|v\cdot\bar{\tau}_{\al}\|_{L_2(S_1)}^2 \le 2
\|f\|^2_{L_2(0,t;L_{6/5}(\Om))}
\\ & + \va\left(\sup_{\tau\le t}\|d\|_{W^{s-1/p}_3(S_2)}\right)\left(\|d\|_{L_2(0,t;W^{1/2}_2(S_2))}^2 +
\|d_t\|_{L_2(0,t;W^{1/6}_{6/5}(S_2))}^2\right) +
\|v(0)\|_{L_2(\Om)}^2 \eal \een where $\va$ is a nonlinear
positive increasing function of its argument and $t\le T.$
\end{theorem}

\begin{theorem}
Assume the compatibility condition (\ref{d}). Let \mbox{$f \in
L_2(kT,(k+1)T;L_{6/5}(\Om)),$} $d_i \in L_{\infty}(R^+;
W^{s-1/p}_p(S_2))$ $ \bigcap L_2(kT,(k+1)T;W^{1/2}_2(S_2)),$ where
$\frac{3}{p} +\frac{1}{3} \le s, p>3 $ or $p=3, s> \frac{4}{3},$
and $d_{i,t} \in L_2(kT,(k+1)T;W^{1/6}_{6/5}(S_2)), i=1,2.$  Let
us assume that $$ \|v(0) \|_{L_2(\Om)} \le A
$$ for some constant $A$ and
\benn 2\int_{kT}^{(k+1)T}\|f\|_{L_{6/5}(\Om)}^2 +
\va\left(\sup_t\|d\|_{W^{s-1/p}_p(S_2)}\right)
\int_{kT}^{(k+1)T}\left(\|d\|_{W^{1/2}_2(S_2)}^2 +
\|d_t\|_{W^{1/6}_{6/5}(S_2)}^2\right)  \\ \le (1-e^{-\nu T})A^2
\eenn for $k \in \N_0,$ where $\va$ is a nonlinear positive
increasing function of its argument. Then there exists a global
weak solution $v$ to (\ref{NS}) such that \benn v\in
V^0_2(\Om\times(kT, (k+1)T)) \ \ \forall k\in \N_0=\N \cup\{0\},
\eenn and  \ben \label{est-k}\bal \|v\|_{V^0_2(\Om \times
(kT,t))}^2
 \leq 2\int_{kT}^t \|f\|^2_{L_{6/5}(\Om)}d\tau + A^2 \\ + \va\left(\sup_{\tau}
 \|d\|_{W^{s-1/p}_p(S_2)}\right)\int_{kT}^t \left(\|d\|_{W^{1/2}_2(S_2)}^2 +
\|d_t\|_{W^{1/6}_{6/5}(S_2)}^2\right)d\tau \eal \een%
for $t \in (kT, (k+1)T]. $
\end{theorem}

The main step in this paper is estimate (\ref{est-1})- see
Lemma~2.2. To derive it, we use the Hopf function (see \cite{L},
\cite{G}) and estimates in weighted Sobolev spaces (see
\cite{RZ1},\cite{RZ2}.) The estimate is such that we can show
global estimate (\ref{est-k}) and prove global existence without
assumption of vanishing of the inflow-outflow and the external
force. The paper makes possible to generalize the result from
\cite{RZ3} into the inflow-outflow case.

\section{Estimates}
\setcounter{equation}{0}

To show the existence theorem, we need to obtain the energy type
estimate and for this purpose, we have to make the Neumann
boundary condition {\rm $(\ref{NS})_5$} homogeneous.

To this end, we extend functions corresponding to inflow and
outflow so that \ben \d_i|_{S_2(a_i)} = d_i,\,\, i= 1,2,\, a_1 =
-a,\, a_2= a \een We introduce the function $\eta$, see \cite{L}.
\benn \eta(\si;\eps,\rho) = \left\{ \begin{array}{lr} 1 & 0 \le
\si \le \rho e^{-1/\eps}\equiv r, \\ -\eps
\ln\dis{\frac{\si}{\rho}} & r < \si\le \rho, \\ 0 & \rho < \si <
\infty.
\end{array} \right.
\eenn

We calculate \benn \frac{d\eta}{d\si} = \eta'(\si;\eps,\rho) =
\left\{
\begin{array}{lr} 0 & 0 < \si \le r, \\
-\dis{\frac{\eps}{\si}} & r < \si\le \rho, \\ 0 & \rho < \si <
\infty.
\end{array} \right.
\eenn so that $|\eta'(\si;\eps,\rho)| \le \dis{\frac{\eps}{\si}}.$
We define functions $\eta_i$ on the neighborhood of $S_2$ (inside
$\Om$): \benn \eta_i= \eta(\si_i;\eps,\rho),\  i=1,2,\eenn where
$\si_i$ denote local coordinates defined on small neighborhood of
$S_2(a_i):$ \benn \si_1 = a+x_3, \ \si_2= a-x_3 \eenn and we set
\ben \label{al} \bal \al & = \sum_{i=1}^2 \d_i \eta_i,
\\ b & = \al \bar{e}_3, \ \bar{e}_3= (0,0,1). \eal \een
We construct function $u$ so that \ben \label{u} u= v-b. \een
Therefore, \benn \Div u & = & - \Div b = -\al_{x_3}\quad {\rm in}
\ \ \Om, \\ u\cdot \bar{n} & = & 0 \quad {\rm on} \ \ S. \eenn
Then, the boundary condition for $u$ is homogeneous. The
compatibility condition takes the form \benn \int_{\Om} \al,_{x_3}
dx = -\int_{S_2(-a)} \al|_{x_3=-a} dS_2 + \int_{S_2(a)}
\al|_{x_3=a}dS_2 = 0 \eenn We define function $\va$ as a solution
to the Neumann problem \ben \bal \label{va} \De \va & = - \Div b
\quad {\rm in} \ \ \Om, \\ \bar{n}\cdot\nb \va & = 0 \quad {\rm
on} \ \ S,\\ \int_{\Om} \va dx & = 0. \eal \een Next, we set \ben
\label{w} w = u - \nb \va = v - (b+\nb \va) \equiv v - \de. \een

Consequently, $(w,p)$ is a solution to the following problem \ben
\label{NS-w} \bal w_{t}+w\cdot\nabla w + w\cdot\nabla \de + \de
\cdot\nabla w -\Div \T(w,p) & \\  =  f - \de_t -\de \cdot\nabla
\de + \nu\Div \D(\de) = F(\de,t) \quad &{\rm in}\ \
 \Omega^T,\\
\Div w=0\quad &{\rm in}\ \ \Omega^T,\\
 w\cdot\bar n=0\quad &{\rm
on}\  S^T ,\\
\nu \bar n\cdot\D(w)\cdot\bar\tau_\alpha + \ga w \cdot
\bar\tau_{\al} & \\ = - \nu \bar n\cdot\D(\de)\cdot\bar\tau_\alpha
- \ga \de \cdot \bar\tau_{\al} = B_{1\al}(\de) ,\ \
\alpha=1,2,\quad
 &{\rm on}\ \ S_1^T,\\
  \bar n\cdot\D(w)\cdot\bar\tau_\alpha = -\bar n\cdot\D(\de)\cdot\bar\tau_\alpha= B_{2\al}(\de),\ \
\alpha=1,2,\quad
 &{\rm on}\ \ S_2^T,\\
w\big|_{t=0}=v(0) - \de(0) = w(0)\quad &{\rm in}\ \ \Omega, \eal
\een where $\Div \de =0.$ Moreover, we set \benn \bar{n}|_{S_1} =
\frac{(\va,_{x_1}, \va,_{x_2}, 0)}{\sqrt{\va,_{x_1}^2 +
\va,_{x_2}^2}},\  \bar{\tau}_1|_{S_1}= \frac{(-\va,_{x_2},
\va,_{x_1}, 0)}{\sqrt{\va,_{x_1}^2 + \va,_{x_2}^2}},\
\bar{\tau}_2|_{S_1}=(0,0,1) = \bar{e}_3, \\ \bar{n}|_{S_2(-a)} = -
\bar{e}_3, \ \bar{n}|_{S_2(a)} = \bar{e}_3,\
\bar{\tau}_1|_{S_2}=\bar{e}_1,\  \bar{\tau}_2|_{S_2}= \bar{e}_2
\eenn where $\bar{e}_1=(1,0,0), \bar{e}_2= (0,1,0).$

\noindent We define a weak solution to the problem (\ref{NS-w})
\begin{definition} \label{weak}
We call $w$ a weak solution to problem (\ref{NS-w}) if for any
sufficiently smooth function $\psi$ such that
$$\Div \psi|_{\Om} =0,\ \  \psi\cdot \bar{n}|_S =0 $$ the
integral equality \benn \int_{\Om^T} w_t\cdot\psi dx dt +
\int_{\Om^T} H(w)\cdot \psi dx dt + \nu \int_{\Om^T}
\D(v)\cdot\D(\psi) dx dt + \ga \sum_{\al=1}^2 \int_{S_1^T}
w\cdot\bar{\tau}_{\al}\psi\cdot \bar{\tau}_{\al} dS_1dt \\ -
\sum_{\al,\si=1}^2 \int_{S_{\si}^T} B_{\si\al}\psi\cdot
\bar{\tau}_{\al}dS_{\si} dt = \int_{\Om^T}F\cdot\psi dxdt \eenn
holds, where $$H(w)= w\cdot\nb w+ w\cdot \nb \de + \de\cdot \nb
w.$$
\end{definition}

\begin{lemma} \label{l-weak}
Assume the compatibility condition (\ref{d}). Assume that $f \in
L_2(0,T;L_{6/5}(\Om)),$ $d_i \in L_{\infty}(0,T; W^{s-1/p}_p(S_2))
\cap L_2(0,T;W^{1/2}_2(S_2)),$ where $\frac{3}{p} +\frac{1}{3} \le
s, p>3 $ or $p=3, s> \frac{4}{3},$ $d_{i,t} \in
L_2(0,T;W^{1/6}_{6/5}(S_2)), i=1,2,$ $w(0) \in L_2(\Om).$ Then for
a weak solution to (\ref{NS-w}) holds \ben \label{est-1} \bal &
\|w\|^2_{V^0_2(\Om^t)} + \ga \sum_{\al=1}^2 \int_0^t
\|w\cdot\bar{\tau}_{\al}\|_{L_2(S_1)}^2 \le 2
\|f\|^2_{L_2(0,t;L_{6/5}(\Om))}
\\ &
+\va(\sup_{\tau}\|d\|_{W^{s-1/p}_p(S_2)})\left(\|d\|_{L_2(0,t;W^{1/2}_2(S_2))}^2
+ \|d_t\|_{L_2(0,t;W^{1/6}_{6/5}(S_2))}^2\right)+
\|w(0)\|_{L_2(\Om)}^2 \eal \een where $t\le T, d=(d_1, d_2)$ and
$\va$ is a nonlinear positive increasing function of its argument.
\end{lemma}
\begin{proof}
We use $\psi=w$ as a test function in a weak solution definition
and apply definition of $F$ to obtain \benn
\frac{1}{2}\frac{d}{dt} \|w\|_{L_2(\Om)}^2 + \int_{\Om}(w\cdot\nb
\de\cdot w +\de\cdot \nb w\cdot w)dx - \int_{\Om}\Div \T(w+\de,
p)\cdot w dx \\ = \int_{\Om}(f -\de_t-\de\cdot \nb\de)\cdot w dx
\eenn

We use boundary conditions (\ref{NS}) on $S_1$ and on $S_2$ to
reformulate the third integral on the l.h.s. of the above
inequality  as follows \benn \int_{\Om} \Div \T(w+\de,p)\cdot w dx=
\Oi \Div[\nu\D(w+\de) - p\I]\cdot w dx \\ = \Oi
\Div[\nu\D(w+\de)]\cdot w dx - \Oi p\cdot\nb w =
\Oi D_{ij}(w+\de)w_{j,x_i}dx \\ = \Oi D_{ij}(w) w_{j,x_i}dx + \Oi D_{ij}(\de) w_{j,x_i}dx
=\frac{1}{2} \Oi |D_{ij}(w)|^2dx + \Oi D_{ij}(\de) w_{j,x_i} dx\eenn
Then, we apply Korn inequality to
have the estimate \ben \label{ineq} \bal & \frac{1}{2}\frac{d}{dt}
\|w\|_{L_2(\Om)}^2 + \nu \|w\|_{H^1(\Om)}^2+ \ga
\sum_{\al=1}^2\|w\cdot\bar{\tau}_{\al}\|_{L_2(S_1)}^2 \\ & \le -
\int_{\Om}(w\cdot\nb \de\cdot w +\de\cdot \nb w\cdot w)dx +
c\sum_{\al=1}^2\|\de\cdot\bar{\tau}_{\al}\|_{L_2(S_1)}^2+
c\|\D(\de)\|^2_{L_2(\Om)} \\& +\int_{\Om}(f -\de_t-\de\cdot
\nb\de)w dx \eal \een Next, we focus on the integral \benn \Oi
\de\cdot \nb w\cdot w dx & = & \Oi (b+\nb \va)\cdot \nb w\cdot w
dx = \Oi b\cdot \nb w\cdot w dx +\Oi \nb \va \cdot \nb w\cdot w dx
\\ & = & I_1 + I_2 \eenn We can estimate $I_1$ by H\"{o}lder
inequality and definition of $b$ \benn |I_1| \le \|\nb
w\|_{L_2(\Om)}\|w\|_{L_6(\Om)} \|b\|_{L_3(\Om)} \le
c\|w\|^2_{H^1(\Om)}\|b\|_{L_3(\widetilde{S}_2(\rho))} \\ \le c
\rho^{1/6} \|w\|^2_{H^1(\Om)}\|b\|_{L_6(\widetilde{S}_2(\rho))}
\le c \rho^{1/6} \|w\|^2_{H^1(\Om)}\|\de \|_{L_6(\Om)} \le c
\rho^{1/6}\|w\|^2_{H^1(\Om)}\|\d\|_{H^1(\Om)} \eenn where \benn
\widetilde{S}_2(\rho) = \{x\in \Om: x_3\in(-a,-a+\rho) \cup
(a-\rho,a) \} = \widetilde{S}_2(\rho, a_1)
\cup\widetilde{S}_2(\rho, a_2). \eenn We estimate $I_2$ as follows
\ben \label{I2} I_2= \Oi\nb\va\cdot\nb w\cdot w dx \le \|\nb
\va\|_{L_3(\Om)}\|w\|_{L_6(\Om)}\|\nb w\|_{L_2(\Om)} \een where
\benn \|\nb \va\|_{L_3(\Om)} & \le & c\|\nb
\va\|_{L_{3,-\mu'}(\Om)} \le c\|\nb_{x_3}\!\nb
\va\|_{L_{3,1-\mu'}(\Om)}\le c\|\va\|_{L^2_{3,1-\mu'}(\Om)}\\ &
\le & c\|\Div b\|_{L_{3, 1-\mu'}(\Om)} \eenn and we denote \benn
\|u\|_{L^k_{p,\mu}(\Om)}=(\sum_{|\al|=k} \int|D^{\al}_xu|^p{\rm
min_{i=1,2} |(dist}(x,S_2(a_i))|^{p\mu}dx)^{1/p}, \mu\in\R, p\in
(1,\infty). \eenn To estimate the last norm, we have used the
result of \cite{RZ2} on Poisson equation in weighted Sobolev
spaces and choose $\dis{\frac{2}{3}} \le 1-\mu' \le 1.$ With $\mu
= 1-\mu'$ we have \benn  c \|\Div b\|_{L_{3,\mu}(\Om)} \le
 c \eps \left(\sum_{i=1}^2
\int_{\widetilde{S}_2(a_i)}|\d_i|^3\frac{\si^{3\mu}_i}{\si_i^3}dx\right)^{1/3}
+ \left(\sum_{i=1}^2 \int_{\widetilde{S}_2(a_i)}|\d_{i,x_3}|^3
|\rho(x)|^{3\mu} dx\right)^{1/3} \\ \le
 c\sum_{i=1}^2
\eps\left(\sup_{x_3}\int_{S_2(a_i)}|\d_i|^3
dx'\int_r^{\rho}\frac{\si^{3\mu}_i}{\si_i^3}d\si_i\right)^{1/3} +
\sum_{i=1}^2\left(\sup_{x_3} \int_{S_2(a_i)}|\d_{i,x_3}|^3 dx'
\int_0^{\rho}\si_i^{3\mu} d\si_i\right)^{1/3}
\\  \le  c\eps \rho^{\mu-2/3}\sup_{x_3}\|\d\|_{L_{3}(S_2)} + c \rho^{\mu
+1/3}\sup_{x_3}\|\d,_{x_3}\|_{L_{3}(S_2)}  \eenn where $\si_i =
{\rm dist}\{S_2(a_i), x\}, x\in S_2(a_i, \rho). $ We note, that
the last bound holds for $\mu >\frac{2}{3}$ since for
$\mu=\frac{2}{3}$ the r.h.s. takes the form \benn
c\sup_{x_3}\|\d\|_{L_{3}(S_2)} + c \rho
\sup_{x_3}\|\d,_{x_3}\|_{L_{3}(S_2)}, \eenn which can not be made small for large $\d.$  Then, \benn I_2 \le c
\left[\eps \rho^{\mu-2/3}\sup_{x_3}\|\d\|_{L_{3}(S_2)} + \rho^{\mu
+1/3}\sup_{x_3}\|\d,_{x_3}\|_{L_{3}(S_2)}\right]
\|w\|^2_{H^1(\Om)} \eenn Next, we consider the term \benn
\Oi(w\cdot\nb \de \cdot w)dx = \Oi(w\cdot\nb b \cdot w)dx +
\Oi(w\cdot\nb \nb \va \cdot w)dx = I_3 +I_4. \eenn For $I_4,$ we
have \benn |I_4| \le \left|\Oi\Div(w\cdot\nb\va\cdot w)dx - \Oi
(w\cdot\nb w \cdot \nb \va)dx\right| \\ \le \int_S|\bar{n}\cdot
\nb\va\cdot w^2|dS + \Oi|\nb\va\cdot (w\cdot\nb w)|dx \le
\Oi|\nb\va\cdot (w\cdot\nb w)|dx \eenn so $I_4$ can be treated in
the same way as $I_2$ and therefore \ben \label{I4} \ \ \ |I_4|
\le c \left[\eps \rho^{\mu-2/3}\sup_{x_3}\|\d\|_{L_{3}(S_2)} +
\rho^{\mu +1/3}\sup_{x_3}\|\d,_{x_3}\|_{L_{3}(S_2)}\right]
\|w\|^2_{H^1(\Om)}.  \een On the other hand, using
$b=\al\bar{e_3}= \sum_{i=1}^2\d_i\eta_i\bar{e_3},$ we find the
bound for $I_3$ \benn |I_3| & \le & |\sum_{i=1}^2
\int_{\widetilde{S}_2(\rho, a_i)} w\cdot \nb(\d_i\eta_i)w_3 dx| \\
& \le & |\sum_{i=1}^2 \int_{\widetilde{S}_2(\rho, a_i)}
(w\cdot\nb\d_i\eta_iw_3 + w\cdot\nb\eta_i\d_iw_3) dx| \\ & \le &
\sum_{i=1}^2 \left(\int_{\widetilde{S}_2(\rho, a_i)}
|w\cdot\nb\d_i\eta_i| |w_3|dx + \int_{\widetilde{S}_2(\rho, a_i)} \eps\left|\frac{w_3}{\si_i} w_3
\d_i\right| d\si_i dx_1dx_2\right) \\ & \le & c \sum_{i=1}^2
\|w\|_{L_6(\widetilde{S}_2(\rho,
a_i))}\|w_3\|_{L_3(\widetilde{S}_2(\rho,
a_i))}\|\nb\d_i\|_{L_2(\widetilde{S}_2(\rho, a_i))}
\\ & & +c\eps
\sum_{i=1}^2\|w_3\|_{L_6(\widetilde{S}_2(\rho,
a_i))}\|\d_i\|_{L_3(\widetilde{S}_2(\rho, a_i))}
\left(\int_{\widetilde{S}_2(\rho, a_i)}dx_1dx_2\int^{\rho}_r
d\si_i\left|\frac{w_3}{\si_i}\right|^2 \right)^{1/2} \\ & \le &
c\rho^{1/6}\sum_{i=1}^2 \|w\|^2_{L_6(\widetilde{S}_2(\rho,
a_i))}\|\nb\d_i\|_{L_2(\widetilde{S}_2(\rho, a_i))}
\\ & & +c\eps\sum_{i=1}^2\|w\|_{L_6(\widetilde{S}_2(\rho,
a_i))}\|\nb w_3\|_{L_2(\widetilde{S}_2(\rho,
a_i))}\|\d_i\|_{L_3(\widetilde{S}_2(\rho, a_i))} \\ & \le &
c(\rho^{1/6} + \eps)\|w\|^2_{H^1(\Om)}\|\d\|_{W^1_3(\Om)}. \eenn
Thus, we can summarize estimates for $I_1-I_4$ to conclude that
nonlinear term in (\ref{ineq}) is bounded by \ben \bal
\label{w-de-w} & \left|\int_{\Om}(w\cdot\nb \de\cdot w +\de\cdot
\nb w\cdot w)dx\right| \\ & \le c \|w\|^2_{H^1(\Om)} \left(
 \eps
\rho^{\mu-2/3}\sup_{x_3}\|\d\|_{L_{3}(S_2)} + \rho^{\mu
+1/3}\sup_{x_3} \|\d,_{x_3}\|_{L_{3}(S_2)}\right. \\ & \left. +(\rho^{1/6} +
\eps)\|\d\|_{W^1_3(\Om)}+ \rho^{1/6}\|\d\|_{H^1(\Om)}\right).
 \eal \een  Next, we examine the second term on the r.h.s. of
(\ref{ineq}): \benn
\sum_{\al=1}^2\|\de\cdot\bar{\tau}_{\al}\|_{L_2(S_1)}^2 & \le &
\sum_{\al=1}^2(\|b\cdot\bar{\tau}_{\al}\|_{L_2(S_1)}^2 + \|\nb
\va\cdot\bar{\tau}_{\al}\|_{L_2(S_1)}^2) \\ & \le &
\|\al\|^2_{L_2(S_1)} + c\|\nb\va \|^2_{W^1_{3/2}(\Om)} \\ & \le &
\sum_{i=1}^2\|d_i\|_{L_2(S_1)}^2 + c\|\Div b \|^2_{L_{3/2}(\Om)}
\\ & \le & c \|\d\|^2_{W^1_{3/2}(\Om)} +c
\sum_{i=1}^2\|\nb(\d_i\eta_i)\|_{L_{3/2}(\Om)}^2 \\ & \le & c
\|\d\|^2_{W^1_{3/2}(\Om)} +c
\sum_{i=1}^2\left(\|\nb\d_i\eta_i\|_{L_{3/2}(\Om)}^2
+\|\d_i\nb\eta_i\|_{L_{3/2}(\Om)}^2\right) \\ & \le & c
\|\d\|^2_{W^1_{3/2}(\Om)} +c \sum_{i=1}^2
\|\d_i\nb\eta_i\|_{L_{3/2}(\Om)}^2.
 \eenn
The last expression we calculate in details: \benn \sum_{i=1}^2
\|\d_i\nb\eta_i\|_{L_{3/2}(\Om)}^2 \le
\eps^2\left[\left(\int_{-a+r}^{-a+\rho}dx_3\int_{S_2(a_1)}dx'
\left|\frac{\d_1}{a+x_3}\right|^{3/2}\right)^{4/3} \right.\\
\left. + \left(\int^{a-r}_{a-\rho}dx_3\int_{S_2(a_2)}dx'
\left|\frac{\d_2}{a-x_3}\right|^{3/2}\right)^{4/3}\right]\\ \le
\eps^2 \left[\sup_{x_3}
\|\d_1\|^2_{L_{3/2}(S_2(a_1))}\left(\int_{-a+r}^{-a+\rho}
\left|\frac{1}{a+x_3}\right|^{3/2}dx_3\right)^{4/3} \right. \\
\left.
 + \sup_{x_3} \|\d_2\|^2_{L_{3/2}(S_2(a_2))}
\left(\int^{a-r}_{a-\rho}
\left|\frac{1}{a-x_3}\right|^{3/2}dx_3\right)^{4/3}\right] \\
\le c\eps^2 \sup_{x_3}
\|\d\|_{L_{3/2}(S_2)}^2\left(\int_r^{\rho}\frac{dy}{y^{3/2}}\right)^{4/3}
\le c\eps^2 \sup_{x_3}
\|\d\|_{L_{3/2}(S_2)}^2\left[\frac{1}{r^{1/2}}-\frac{1}{\rho^{1/2}}
\right]^{4/3} \\ \le c\eps^2 \sup_{x_3}
\|\d\|^2_{L_{3/2}(S_2)}\frac{1}{\rho^{2/3}}[e^{1/2\eps} - 1]^{4/3}
\le c \frac{\eps^2}{\rho^{2/3}} e^{2/3\eps} \sup_{x_3}
\|\d\|_{L_{3/2}(S_2)}^2.
 \eenn
Combining inequalities above, we infer \benn
\sum_{\al=1}^2\|\de\cdot\bar{\tau}_{\al}\|_{L_2(S_1)}^2 \le c
\|\d\|^2_{W^1_{3/2}(\Om)} + c \frac{\eps^2}{\rho^{2/3}}
e^{2/3\eps} \sup_{x_3} \|\d\|_{L_{3/2}(S_2)}^2 \eenn We estimate
also the term \benn \|\D(\de)\|^2_{L_2(\Om)} & \le &
\|\D(b)\|^2_{L_2(\Om)} + \|\D(\nb\va)\|^2_{L_2(\Om)}\\ & \le &
\sum_{i=1}^2 \left(\|\nb\d_i\eta_i\|^2_{L_2(\Om)}+ \|\d_i\nb
\eta_i\|^2_{L_2(\Om)}\right) + \|\nb^2\va\|_{L_2(\Om)}^2 \\ & \le
& \sum_{i=1}^2 \left(\|\nb\d_i\eta_i\|^2_{L_2(\Om)}+ \|\d_i\nb
\eta_i\|^2_{L_2(\Om)}\right) + \|\Div b\|_{L_2(\Om)}^2 \\ & \le &
c\sum_{i=1}^2 \left(\|\nb\d_i\eta_i\|^2_{L_2(\Om)}+ \|\d_i\nb
\eta_i\|^2_{L_2(\Om)}\right) \\ & \le &
c\sum_{i=1}^2\|\d_i\|^2_{W^1_2(\Om)} + \eps^2
c\int_{-a+r}^{-a+\rho}dx_3\int_{S_2(a_1)}dx'
\left|\frac{\d_1}{a+x_3}\right|^{2} \\ & & + \eps^2
\int^{a-r}_{a-\rho}dx_3\int_{S_2(a_2)}dx'
\left|\frac{\d_2}{a-x_3}\right|^{2} \eenn \benn & \le & c
\sum_{i=1}^2\left(\|\d_i\|^2_{W^1_2(\Om)}+ \eps^2 \sup_{x_3}
\|\d_i\|^2_{L_{2}(S_2)}\int_r^{\rho}\frac{dy}{y^2}\right) \\ & \le
& c \sum_{i=1}^2\left[\|\d_i\|^2_{W^1_2(\Om)}+ \eps^2 \sup_{x_3}
\|\d_i\|^2_{L_{2}(S_2)}\left(\frac{1}{r}- \frac{1}{\rho} \right)
\right]  \\ & \le & c \sum_{i=1}^2\left[\|\d_i\|^2_{W^1_2(\Om)}+
\eps^2 \sup_{x_3} \|\d_i\|^2_{L_{2}(S_2)}\frac{1}{\rho}(e^{1/\eps}
-1) \right] \\ & \le & c
\sum_{i=1}^2\left[\|\d_i\|^2_{W^1_2(\Om)}+ \frac{\eps^2}{\rho}
e^{1/\eps} \sup_{x_3} \|\d_i\|^2_{L_{2}(S_2)}\right]. \eenn

Analyzing the last integral on the r.h.s. of (\ref{ineq}) we have
\benn \int_{\Om}(f  -\de_t -\de\cdot \nb\de)w dx &  \le &
\eps_1\|w\|_{L_6(\Om)}^2 +c(1/\eps_1) (\|f\|_{L_{6/5}(\Om)}^2  +
\|\de_t\|_{L_{6/5}(\Om)}^2) \\ && + \left|\int_{\Om} \de\cdot\nb
\de\cdot wdx\right| \eenn%
 We estimate $\|\de_t\|_{L_{6/5}(\Om)}$
as follows \benn \|\de_t\|_{L_{6/5}(\Om)}& = & \|b_t +
\nb\va_t\|_{L_{6/5}(\Om)} \le \|\d_t\|_{L_{6/5}(\Om)}+ \|\Div
b_t\|_{L_{6/5}(\Om)} \\ & \le & \|\d_t\|_{L_{6/5}(\Om)}+ \|\nb
\d_t\|_{L_{6/5}(\Om)}+ \|\d_t\nb \eta \|_{L_{6/5}(\Om)}
 \\ & \le & \|\d_t\|_{W^1_{6/5}(\Om)} +  \eps
 \sup_{x_3}\|\d_t\|_{L_{6/5}(S_2)}\left(\int_r^{\rho}\frac{dx_3}{x_3^{6/5}}\right)^{5/6}
\\ & \le & \|\d_t\|_{W^1_{6/5}(\Om)} +  \eps \frac{1}{\rho^{1/6}}
 e^{1/6\eps}
 \sup_{x_3}\|\d_t\|_{L_{6/5}(S_2)} \eenn
 since \benn
 \left(\int_r^{\rho}\frac{dx_3}{x_3^{6/5}}\right)^{5/6}=
 \left(\frac{1}{r^{1/5}} - \frac{1}{\rho^{1/5}}\right)^{5/6} =
 \frac{1}{\rho^{1/6}}\left(e^{1/5\eps} -1\right)^{5/6} \eenn
Finally, we examine \benn \left|\int_{\Om} \de\cdot\nb \de\cdot
wdx\right| \le \|\nb \de\|_{L_2(\Om)}\|w\|_{L_6(\Om)}
\|\de\|_{L_3(\Om)}\le  \eps_2 \|w\|_{L_6(\Om)}^2 + c(1/\eps_2) \|
\de\|_{W^1_2(\Om)}^4 \\ \le \eps_2 \|w\|_{L_6(\Om)}^2 +
c(1/\eps_2)\left(\|\d\|_{W^1_2(\Om)}^4 + \frac{\eps^4}{\rho^2}
e^{2/\eps}\sup_{x_3}\|\d\|_{L_2(S_2)}^4\right)\eenn

We summarize above estimates to rewrite (\ref{ineq}) as follows
\ben \label{est} \bal & \frac{1}{2}\frac{d}{dt} \|w\|_{L_2(\Om)}^2
+ \nu \|w\|_{H^1(\Om)}^2+ \ga
\sum_{\al=1}^2\|w\cdot\bar{\tau}_{\al}\|_{L_2(S_1)}^2 \\ & \le
\|w\|_{H^1(\Om)}^2 \left[ \eps
\rho^{\mu-2/3}\sup_{x_3}\|\d\|_{L_{3}(S_2)} + \rho^{\mu
+1/3}\sup_{x_3}\|\d,_{x_3}\|_{L_{3}(S_2)}\right. \\ & \left.
+(\rho^{1/6} + \eps)\|\d\|_{W^1_3(\Om)}+
\rho^{1/6}\|\d\|_{H^1(\Om)}+ \eps_1 +\eps_2 \right] \\ & +
\|f\|_{L_{6/5}(\Om)}^2 + \|\d\|^2_{L_2(\Om)} +
\|\d\|^4_{W_2^1(\Om)}+ \|\d\|^2_{W_2^1(\Om)}
\\ & + \|\nb\d\|^2_{L_{6/5}(\Om)} + \|\nb \d\|_{L_2(\Om)}^4 +
\|\d\|^2_{W^1_{3/2}(\Om)} + \|\d_t\|^2_{W^1_{6/5}(\Om)} \\ & +
\frac{\eps^2}{\rho} e^{1/\eps} \sup_{x_3} \|\d\|^2_{L_{2}(S_2)} +
\frac{\eps^4}{\rho^2} e^{2/\eps}\sup_{x_3}\|\d\|_{L_2(S_2)}^4
\\ & +\frac{\eps^2}{\rho^{2/3}} e^{2/3\eps} \sup_{x_3}
\|\d\|_{L_{3/2}(S_2)}^2  + \frac{\eps^2}{\rho^{1/3}}
 e^{1/3\eps}
 \sup_{x_3}\|\d_t\|^2_{L_{6/5}(S_2)} \eal \een
We apply Sobolev anisotropic imbedding (see \cite{BIN}, Ch.3,
Section 10) to estimate $\sup_{x_3}\|\d\|_{L_{3}(S_2)}$ and
$\sup_{x_3}\|\d,_{x_3}\|_{L_{3}(S_2)}$ with some $W^s_p$ norm and
calculate \benn 2\left(\frac{1}{p} - \frac{1}{3}\right)\frac{1}{s}
+ \frac{1}{p}\cdot\frac{1}{s} + \frac{1}{s} \le 1 \quad {\rm for}
\quad p>3 \\
2\left(\frac{1}{p} - \frac{1}{3}\right)\frac{1}{s} +
\frac{1}{p}\cdot\frac{1}{s} + \frac{1}{s} < 1 \quad {\rm for}
\quad p=3 \eenn Then, \ben \label{s-p} \frac{3}{p} +\frac{1}{3}
\le s \quad {\rm for}\quad p>3 \quad {\rm or} \quad p=3, \,
s>\frac{4}{3}. \een We set $\mu > \frac{2}{3},$ then since $\rho
<1,$ we observe that $ \rho^{\mu+1/3} \le \rho^{1/6}.$ Then \benn
\eps \rho^{\mu-2/3}\sup_{x_3}\|\d\|_{L_{3}(S_2)}  +  \rho^{\mu
+1/3}\sup_{x_3}\|\nb\d\|_{L_{3}(S_2)} +(\rho^{1/6} +
\eps)\|\d\|_{W^1_3(\Om)}  +  \rho^{1/6}\|\d\|_{H^1(\Om)} \\
\le (\eps\rho^{\mu-2/3} + \rho^{\mu +1/3} + 2 \rho^{1/6}
+\eps)\|\d\|_{W^s_p(\Om)} \le (2\eps +
3\rho^{1/6})\|\d\|_{W^s_p(\Om)} \eenn
 We put \ben
\label{eps} \bal \eps & = \frac{\nu}{15\|\d\|_{W^s_p(\Om)}},\\
\rho^{1/6} & = \frac{\nu}{15\|\d\|_{W^s_p(\Om)}},\\ \eps_1+\eps_2
& = \frac{\nu}{6},  \eal \een with $p, s$ satisfying (\ref{s-p}).
Therefore, \benn \eps \rho^{\mu-2/3}\sup_{x_3}\|\d\|_{L_{3}(S_2)}
+ \rho^{\mu +1/3}\sup_{x_3}\|\nb\d\|_{L_{3}(S_2)} \\ +(\rho^{1/6}
+ \eps)\|\d\|_{W^1_3(\Om)}  +  \rho^{1/6}\|\d\|_{H^1(\Om)} +
\eps_1 + \eps_2  \le \frac{\nu}{2} \eenn and formula (\ref{est})
assumes the form \benn  \frac{d}{dt} \|w\|_{L_2(\Om)}^2 & + & \nu
\|w\|_{H^1(\Om)}^2+ \ga
\sum_{\al=1}^2\|w\cdot\bar{\tau}_{\al}\|_{L_2(S_1)}^2 \\ & \le &
2\|f\|_{L_{6/5}(\Om)}^2 + \va(\|\d\|_{W^1_2(\Om)})
(\|\d\|_{W^1_2(\Om)}^2 + \|\d_t\|_{W^1_{6/5}(\Om)}^2) \\ & & +
\va(\|\d\|_{W^s_p(\Om)})\left(\sup_{x_3}
\|\d\|^2_{L_{2}(S_2)}+\sup_{x_3}\|\d_t\|^2_{L_{6/5}(S_2)}\right)
 \eenn where $\va$ is a nonlinear positive increasing function of
 its argument.
We use Sobolev imbedding  \benn
\sup_{x_3} \|\d\|_{L_{2}(S_2)}\le c \|\d\|_{W^1_2(\Om)}, \\
 \sup_{x_3}\|\d_t\|_{L_{6/5}(S_2)} \le c \|\d_t\|_{W^1_{6/5}(\Om)}
\eenn and hence \ben \label{est-1} \bal & \frac{d}{dt}
\|w\|_{L_2(\Om)}^2 + \nu \|w\|_{H^1(\Om)}^2+ \ga
\sum_{\al=1}^2\|w\cdot\bar{\tau}_{\al}\|_{L_2(S_1)}^2 \\ & \le
2\|f\|_{L_{6/5}(\Om)}^2 + \va(\|\d\|_{W^s_p(\Om)})
\left(\|\d\|_{W^1_2(\Om)}^2 + \|\d_t\|_{W^1_{6/5}(\Om)}^2\right)
\eal \een Integrating (\ref{est-1}) with respect to time we obtain
\ben \label{est-t} \bal & \|w\|^2_{V^0_2(\Om^t)} + \ga
\sum_{\al=1}^2 \int_0^t \|w\cdot\bar{\tau}_{\al}\|_{L_2(S_1)}^2 dt
\le 2 \|f\|^2_{L_2(0,t;L_{6/5}(\Om))}
\\ &
+\va(\sup_{\tau}\|\d\|_{W^s_p(\Om)})\left(\|\d\|_{L_2(0,t;W^1_2(\Om))}^2
+ \|\d_t\|_{L_2(0,t;W^1_{6/5}(\Om))}^2\right)+
\|w(0)\|_{L_2(\Om)}^2, \eal \een where $\frac{3}{p} +\frac{1}{3}
\le s, p>3 $ or $p=3, s>\frac{4}{3}.$

\end{proof}

\section{Weak solutions to (\ref{NS-w})}
\setcounter{equation}{0}

In this section, we use the Galerkin method to prove the existence
of weak solutions to the problem (\ref{NS-w}). We follow ideas
from [L], chapter 6, section 7. Namely, we introduce the
sequence of approximating functions $w_N$ given as %
\benn w^N(x,t) = \sum_{k=1}^N C_{kN}(t) a^k(x), \eenn where
$\{a^k\}_{k=1}^{\infty}$ is the system of orthogonal functions in
$L^2(\Om)\bigcap J_2^0(\Om).$ Here, $J^0_2(\Om)= \{f\in H^1(\Om) :
{\rm div} f=0\}$ and $\{a^k\}_{k=1}^{\infty}$ is the fundamental
system in $H^1(\Om)$ with $\sup_{x\in \Om}|a^k(x)| <\infty,
\sup_{x\in \pa\Om} |a^k(x)|<\infty.$ The coefficients $C_{kN}(0)$
are defined by \benn C_{kN}|_{t=0} = (w_0, a_k), \quad
k=1,\ldots,N, \eenn and the
function $w^N$ satisfy the following system with test functions $a^k$:%
\benn%
 & &  \left\{ \Oi \left( \frac{1}{2}\frac{d}{dt} w^N a^k + w^N\cdot\nb w^N a^k+ \de
\cdot \nb w^N \cdot a^k + w^N \cdot\nb\de\cdot a^k +\nu
\mathbb{D}(w^N)\mathbb{D}(a^k)\right)dx \right.
\\& & \left.+ \ga \int_{S_1} w^N \cdot\bar{\tau}_j a^k \bar{\tau}_j dS_1 \right \}=
\left(\sum_{j,\si=1}^2 \int_{S_{\si}}B_{\si j}a^k
\cdot\bar{\tau}_j\,
dS_{\si} + \Oi  F\cdot a^k dx \right)\eenn %
for $k=1,\ldots,N.$ Then, $w^N$ would be the weak solution to
(\ref{NS-w}).

 With $(f,g)= \int_{\Om} f g dx $ and
$(f,g)_{S}= \int_{S} fg dS$ this can be rewritten as:
\benn%
& &  \biggl\{(w^N_{t},a^k) + (w^N \cdot\nb w^N,a^k) +(\de \cdot
\nb w^N,a^k) + (w^N\cdot\nb\de,a^k)
\\ & & +\nu (\mathbb{D}(w^N),\mathbb{D}(a^k))
 + \ga (w^N\cdot\bar{\tau}_j,a^k\cdot
\bar{\tau}_j)_{S_1}\biggr\} = \\ & & \left[\sum_{\si,j=1}^2
(B_{\si j},a^k \cdot\bar{\tau}_j)_{S_{\si}} + (F,a^k)\right],
\quad k=1,\ldots,N.
\eenn%
Thus,%
\ben \label{gal}\bal%
 \left(\frac{d}{dt}w^N,a^k\right) & +  (w^N \cdot\nb w^N,a^k)
+(\de \cdot \nb w^N,a^k) + (w^N\cdot\nb\de,a^k) \\ &   + \nu
(\mathbb{D}(w^N),\mathbb{D}(a^k)) + \ga
(w^N\cdot\bar{\tau}_j,a^k\cdot \bar{\tau}_j)_{S_1}\\& =
\sum_{j,\si=1}^2 (B_{\si j},a^k \cdot\bar{\tau}_j)_{S_{\si}} +
(F,a^k), \quad k=1,\ldots,N. \eal \een The above equations are in
fact a system of ordinary differential equations for the functions
$C_{kN}(t).$ The properties
of the sequence $a^k$ imply%
\benn |w^N(x,t)|_{2,\Om}^2 = \sum_{k=1}^N C_{kN}^2(t). \eenn %
On the other hand, we can obtain the a priori bounds for the
approximative
solutions $w^N$ of the same form as (\ref{est-t}): %
\ben \label{wN-apriori} \bal  |w^N|_{V^0_2(\Om^T)}^2 & =
\sup_{0\leq t\leq T}|w^N|_{2,\Om} + \int_0^T |\nb w^N|_{2,\Om}dt \\
& \leq \int_0^T \|f\|^2_{L_{6/5}(\Om)} + \va(\sup_{0\leq t\leq T}
\|\d\|_{W^s_p(\Om)})\int_0^T\left(\|\d\|_{W^1_2(\Om)}^2 +
\|\d_t\|_{W^1_{6/5}(\Om)}^2\right)dt \\ & +
\|w^N(0)\|_{L_2(\Om)}^2 \leq C, \eal \een where $\frac{3}{p}
+\frac{1}{3} \le s, p>3 $ or $p=3, s>\frac{4}{3}.$ Therefore,
$\sup_{0\leq t\leq T} |C_{kN}(t)|$ is bounded on $[0,T]$ and $w^N$
are well defined for all times $t.$

 Let us define now
$\psi_{N,k} \equiv (w^N(x,t),a^k(x)).$ This sequence is uniformly
bounded by (\ref{wN-apriori}). We can also show that it is
equicontinuous. Namely, we integrate (\ref{gal}) with respect
to $t$ from $t$ to $t+\De t$ to obtain%
\benn & & |\psi_{N,k}(t+\De t) - \psi_{N,k}(t)| \leq
\sup_{x\in\Om} |a^k(x)| \int_t^{t+\De t}\left(|w^N \cdot\nb
w^N|_{2,\Om} +|\de \cdot \nb w^N|_{2,\Om}\right. \\ & & \left. +
|w^N\cdot\nb\de|_{2,\Om}+ |F|_{2,\Om}\right)dt +\nu |\nb
a^k|_{2,\Om}\int_t^{t+\De t} |\nb w^N|_{2,\Om} dt \\& & + \ga
\sup_{x\in S} |a^k(x)| \int_t^{t+\De t}
\left(|w^N\cdot\bar{\tau}_j|_{2,S_1} + \sum_{j,\si=1}^2 |B_{\si
j}|_{2,S_{\si}}\right)dt \\ & & \leq \sup_{x\in\Om} |a^k(x)|
\sqrt{\De t}\left( \sup_{x\in\Om} |w^N|_{2,\Om} (|\nb
w^N|_{2,\Om^T} + |\nb \de|_{2,\Om^T})+
\sup_{x\in\Om}|\de|_{2,\Om}|\nb w^N|_{2,\Om^T} \right)\\
& & + \sup_{x\in\Om} |a^k(x)| \int_t^{t+\De t} |F|_{2,\Om}dt +\nu
|\nb a^k|_{2,\Om}\sqrt{\De t}|\nb w^N|_{2,\Om^T}\\ & & + \ga
\sup_{x\in S} |a^k(x)|\left(\sqrt{\De t}|\nb w^N|_{2,\Om^T}
+\int_t^{t+\De t} \sum_{j=1}^2 |B_{j}|_{2,S})\right)dt \\ & & \leq
C(k)\left(\sqrt{\De t} + \int_t^{t+\De t} (|F|_{2,\Om}+
\sum_{j=1}^2
|B_{j}|_{2,S})dt\right). \eenn%
We can see that for given $k$ and $N\geq k$ the r.h.s. tends to
zero as $\De t\ra 0$ uniformly in $N.$ Thus, it is possible to
choose a subsequence $N_m$ such that $\psi_{N_m,k}$ converges with
$m\ra\infty$ uniformly to some continuous function $\psi_k$
for any given $k.$ Since the limit function $w$ is defined as %
\benn w(x,t) = \sum_{k=1}^{\infty} \psi_k (t) a^k(x), \eenn then
we conclude that $(w^{N_m} - w,\psi(x))$ tends to zero as
$m\ra\infty$ uniformly with respect to $t\in [0,T]$ for any
$\psi\in J_2^0(\Om)$ and $w(x,t)$ is continuous in $t$ in weak
topology. Moreover, estimates (\ref{wN-apriori}) remain true for
the limit function $w.$

We will show that $\{w^{N_m}\}$ converges strongly in
$L^2(\Om^T).$ To this end, we need to apply the following version
of the Friedrichs lemma: for any $\eps >0,$ there exists such
$N_{\eps}$ that for any $u\in W_2^1(\Om)$ the following inequality
holds:%
\benn ||u||_{2,\Om}^2 \leq \sum_{k=1}^{N_{\eps}} (u,a^k) + \eps
||\nb u||_{2,\Om}^2. \eenn%
This in terms of $u= w^{N_m} - w^{N_l}$ reads%
\benn ||w^{N_m} - w^{N_l}||_{2,\Om^T}^2 \leq \sum_{k=1}^{N_{\eps}}
\int_0^T (w^{N_m} - w^{N_l},a^k) dt + \eps
||\nb w^{N_m} - \nb w^{N_l}||_{2,\Om^T}^2. \eenn%
By (\ref{wN-apriori}), we have $$||\nb w^{N_m} - \nb
w^{N_l}||_{2,\Om^T}^2 \leq 2C^2$$ for some constant $C.$ The first
integral on the r.h.s. for given number $N_{\eps}$ can be
arbitrarily small if only $m$ and $l$ are sufficiently large, so
it tends to zero as $m,l\ra \infty.$ Therefore, $\{w^{N_m} \} $
converges strongly in $L^2(\Om^T).$

We summarize the above convergence properties of the sequence
$\{w^{N_m}\}:$

(i) $\{w^{N_m}\} \ra w $ strongly in $L^2(\Om^T)$ for some $w,$

(ii) $\{w^{N_m}\} \ra w$ weakly in $L^2(\Om)$ uniformly with
respect to $t \in [0,T],$

(iii) $\nb \{w^{N_m}\} \ra \nb w$ weakly in $L^2(\Om^T).$

With given $\Phi^k = \sum_{j=1}^k d_j(t) a^j(x)$, the sequence
$\{w^{N_m}\}$ satisfy the identities: %
\benn & & \int_{\Om} \left(\frac{d}{dt}w^{N_m}\Phi^k + (w^{N_m}
\cdot\nb w^{N_m} +\de \cdot \nb w^{N_m} +
w^{N_m}\cdot\nb\de)\Phi^k  +\nu
\mathbb{D}(w^{N_m})\mathbb{D}(\Phi^k)\right)dx \\& & + \ga
\int_{S_1} w^{N_m}\cdot\bar{\tau}_j\Phi^k\cdot \bar{\tau}_j dS_0 =
\sum_{\si,j=1}^2 \int_{S_{\si}}B_{\si j}\Phi^k \cdot\bar{\tau}_j
dS_{\si} + \int_{\Om}F\Phi^k dx. \eenn

Then, we can pass to the limit with $m\ra \infty$ to obtain the
identity for $w$. Conditions ${\rm div} w^N = 0, w^N \cdot \bar{n}
|_{S^T} = 0$ stay true for the limit function $w$ as well.

It remains to consider the limit $\lim_{t\ra 0} w(x,t).$ We note,
that $w^{N_m}$ satisfy the relation (\ref{ineq}) (if we
use the test function $w^{N_m}$). This yields%
\benn |w^{N_m}|_{2,\Om} \leq |w_0|_{2,\Om} + \int_0^t
(|F|_{2,\Om} +|B|_{2,S})dt. \eenn%
In the limit $m\ra \infty$ we obtain %
\benn |w|_{2,\Om} \leq |w_0|_{2,\Om} + \int_0^t
(|F|_{2,\Om} +|B|_{2,S})dt \eenn%
which implies $$\overline{\lim}_{t\ra 0} |w|_{2,\Om} \leq
|w_0|_{2,\Om}.$$ On the other hand, since $w^{N_m}$ tends to $w$
as $m\ra\infty$, we have $ |w^{N_m}-w_0|_{2,\Om} \ra 0$.
Therefore, $|w^{N_m}-w_0| \ra 0$ weakly in $L^2(\Om)$ as $t\ra 0$
and
$$|w_0|_{2,\Om} \leq \underline{\lim}_{t\ra 0} |w|_{2,\Om}.$$
We conclude that the limit $\lim_{t\ra 0} |w|_{2,\Om}$ exists and
is equal to $|w_0|_{2,\Om}$ where the convergence is strong - in
the norm $L^2(\Om).$

Consequently, we have proved the following result.

\begin{lemma}
Let the assumptions of Lemma~\ref{l-weak} be satisfied. Then there
exists a weak solution $w$ to problem (\ref{NS-w}) such that $w$
is weakly continuous  with respect to $t$ in $L^2(\Om)$ norm and
$w$  converges to $w_0$ as $t\ra 0$ strongly in $L^2(\Om)$ norm.
\end{lemma}

Since $v=w-\de$ we conclude the analogous existence result for $v$
formulated in Theorem~1.

\section{Global solutions to (\ref{NS-w})}
\setcounter{equation}{0}

To obtain a global estimate we write (\ref{est-1}) in the form
\benn  \frac{d}{dt} \|w\|_{L_2(\Om)}^2 + \nu \|w\|_{L_2(\Om)}^2
\le 2\|f\|_{L_{6/5}(\Om)}^2 + \va(\|\d\|_{W^s_p(\Om)})
\left(\|\d\|_{W^1_2(\Om)}^2 + \|\d_t\|_{W^1_{6/5}(\Om)}^2\right),
\eenn where $\frac{3}{p} +\frac{1}{3} \le s, p>3 $ or $p=3, s>
\frac{4}{3}.$ Hence \benn  \frac{d}{dt} \left( \|w\|_{L_2(\Om)}^2
e^{\nu t} \right) \le 2\|f\|_{L_{6/5}(\Om)}^2 e^{\nu t} +
\va(\|\d\|_{W^s_p(\Om)}) \left(\|\d\|_{W^1_2(\Om)}^2 +
\|\d_t\|_{W^1_{6/5}(\Om)}^2\right) e^{\nu t} \eenn Integrating
with respect to time from $t_1$ to $t_2$ yields \benn
\|w(t_2)\|_{L_2(\Om)}^2 e^{\nu t_2} & \le & 2
\int_{t_1}^{t_2}\|f\|_{L_{6/5}(\Om)}^2 e^{\nu t}dt
+\|w(t_1)\|_{L_2(\Om)}^2 e^{\nu t_1} \\ & & +
\va(\sup_t\|\d\|_{W^s_p(\Om)})
\int_{t_1}^{t_2}\left(\|\d\|_{W^1_2(\Om)}^2 +
\|\d_t\|_{W^1_{6/5}(\Om)}^2\right) e^{\nu t}dt.
 \eenn Thus, \benn
\|w(t_2)\|_{L_2(\Om)}^2 \le  2 e^{-\nu
t_2}\int_{t_1}^{t_2}\|f\|_{L_{6/5}(\Om)}^2 e^{\nu t}dt + \|w(t_1)\|_{L_2(\Om)}^2 e^{-\nu(t_2-t_1)} \\
+ \va(\sup_t\|\d\|_{W^s_p(\Om)}) e^{-\nu
t_2}\int_{t_1}^{t_2}\left(\|\d\|_{W^1_2(\Om)}^2 +
\|\d_t\|_{W^1_{6/5}(\Om)}^2\right)e^{\nu t} dt \eenn and this
implies \ben \bal \|w(t_2)\|_{L_2(\Om)}^2 \le
2\int_{t_1}^{t_2}\|f\|_{L_{6/5}(\Om)}^2dt
 + \|w(t_1)\|_{L_2(\Om)}^2 e^{-\nu(t_2-t_1)} \\
+ \va(\sup_t\|\d\|_{W^s_p(\Om)})
\int_{t_1}^{t_2}\left(\|\d\|_{W^1_2(\Om)}^2 +
\|\d_t\|_{W^1_{6/5}(\Om)}^2\right)dt \eal \een

Setting $t_1=0$ and $t_2=t\in R_+$ we obtain the global estimate
\ben \bal \|w(t)\|_{L_2(\Om)}^2 \le  2
\int_0^t\|f\|_{L_{6/5}(\Om)}^2 d\tau + \|w(0)\|_{L_2(\Om)}^2 e^{-\nu t}\\
+ \va(\sup_t\|\d\|_{W^s_p(\Om)})
\int_0^t\left(\|\d\|_{W^1_2(\Om)}^2 +
\|\d_t\|_{W^1_{6/5}(\Om)}^2\right)d\tau  \eal \een Let $k \in \N.$
Integrating (\ref{est-1}) with respect to time from $kT$ to $t\in
(kT, (k+1)T]$ we get  \ben \label{est-k1}\bal |w|_{V^0_2(\Om
\times (kT,t))}^2
 \leq 2\int_{kT}^t \|f\|^2_{L_{6/5}(\Om)}d\tau + \|w(kT)\|_{L_2(\Om)}^2 \\ + \va(\sup_{\tau}
 \|\d\|_{W^s_p(\Om)})\int_{kT}^t \left(\|\d\|_{W^1_2(\Om)}^2 +
\|\d_t\|_{W^1_{6/5}(\Om)}^2\right)d \tau \eal \een Therefore, \ben
\label{est-k}\bal |v|_{V^0_2(\Om \times (kT,t))}^2
 \leq 2\int_{kT}^t \|f\|^2_{L_{6/5}(\Om)}d\tau + \|v(kT)\|_{L_2(\Om)}^2 \\ + \va(\sup_{\tau}
 \|\d\|_{W^s_p(\Om)})\int_{kT}^t \left(\|\d\|_{W^1_2(\Om)}^2 +
\|\d_t\|_{W^1_{6/5}(\Om)}^2\right)d\tau \eal \een%
  We have also \ben \bal \|v(T)\|_{L_2(\Om)}^2 \le 2
\int_0^t\|f\|_{L_{6/5}(\Om)}^2d\tau + \|v(0)\|_{L_2(\Om)}^2 e^{-\nu T}\\
+ \va(\sup_t\|\d\|_{W^s_p(\Om)})
\int_0^t\left(\|\d\|_{W^1_2(\Om)}^2
+ \|\d_t\|_{W^1_{6/5}(\Om)}^2\right) d\tau  \eal \een %
We set $\mu_1= e^{-\nu T}.$ Let as assume that $$ \| v(0)
\|_{L_2(\Om)} \le A
$$ for some constant $A$ and
$$2\int_0^t\|f\|_{L_{6/5}(\Om)}^2 d\tau + \va(\sup_t\|\d\|_{W^s_p(\Om)}) \int_0^t\left(\|\d\|_{W^1_2(\Om)}^2
+ \|\d_t\|_{W^1_{6/5}(\Om)}^2\right)d\tau \le (1-e^{-\nu T})A^2$$
Thus,
$$ \|v(T)\|_{L_2(\Om)} \le A $$ so we can control the initial
condition for the next time step. This can be repeated for
intervals $(kT, (k+1)T)$. Then by (\ref{est-k}) we can prove
global existence of weak solution such that \benn v\in
V^0_2(\Om\times(kT, (k+1)T)) \ \ \forall k\in \N_0=\N \cup\{0\},
\eenn so we conclude Theorem~2.


\bigskip

\begin{thebibliography}{GK2}
\bibitem[BIN]{BIN}
             {\sc O. V. Besov,} {\sc V. P. Il'in} and {\sc S. M. Nikol'skii},
             {\it Integral representations of functions and imbedding theorems.}\/, Vol. I. Translated from the Russian. Scripta Series in Mathematics, New York-Toronto, Ont.-London, 1978. viii+345 pp.
\bibitem[G]{G}
             {\sc G. P. Galdi},
             {\it An introduction to the mathematical theory of the Navier-Stokes equations.}\/, Vol. II. Nonlinear steady problems. Springer Tracts in Natural Philosophy, 39. Springer-Verlag, New York, 1994. xii+323 pp.
\bibitem[K1]{K1}
              {\sc P. Kacprzyk},
              {\it Global regular nonstationary flow for the
               Navier-Stokes equations in a~cylindrical pipe}\/,
               Appl. Math. 34(3)(2007), 289--307.
\bibitem[K2]{K2}
              {\sc P. Kacprzyk},
              {\it Global existence for the inflow-outflow problem
              for the Navier-Stokes equations in a cylinder}\/,
              Appl. Math 36(2) (2009), 195--212.
\bibitem[L]{L}
              {\sc O. A. Ladyzhenskaya},
              {\it{Mathematical Theory of Viscous Incompressible
              Flow}}\/,
              Nauka, Moscow 1970 (in Russian).
\bibitem[RZ1]{RZ1}
              {\sc J. Renc{\l}awowicz} and {\sc W.M.
              Zaj\c{a}czkowski},
              {\it Existence of solutions to the Poisson
               equation in $L_2$-weighted spaces.}\/, to appear in Appl. Math.
\bibitem[RZ2]{RZ2}
              {\sc J. Renc{\l}awowicz} and {\sc W.M.
              Zaj\c{a}czkowski},
              {\it Existence of solutions to the Poisson
               equation in $L_p$-weighted spaces.}\/, to appear in
               Appl. Math.
\bibitem[RZ3]{RZ3}
              {\sc J. Renc{\l}awowicz} and {\sc W.M.
              Zaj\c{a}czkowski},
              {\it Large time regular solutions to the Navier-Stokes equations in
cylindrical domains}\/, Topol. Methods Nonlinear Anal. 32 (2008), 69-87.
\bibitem[Z]{Z}
              {\sc W.M. Zaj\c{a}czkowski},
              {\it Global regular nonstationary flow for the
               Navier-Stokes equations in a~cylindrical pipe}\/,
               TMNA 26(2005), 221-286.
\end{thebibliography}
\end{document}